\documentclass[10pt,a4paper,reqno]{amsart}

\usepackage[T1]{fontenc}
\usepackage{lmodern}
\usepackage[british]{babel}
\usepackage{microtype}
\usepackage{needspace}
\usepackage{amsmath,amssymb,mathtools}
\usepackage[mathscr]{euscript}
\usepackage[hidelinks]{hyperref}
\allowdisplaybreaks[2]
\numberwithin{equation}{section}
\theoremstyle{plain}
\newtheorem{theorem}{Theorem}[section]
\newtheorem{proposition}[theorem]{Proposition}
\newtheorem{lemma}[theorem]{Lemma}
\theoremstyle{definition}
\newtheorem{definition}[theorem]{Definition}
\theoremstyle{remark}
\newtheorem{remark}[theorem]{Remark}

\DeclareMathOperator{\Ric}{Ric}
\DeclareMathOperator{\Span}{span}
\renewcommand{\Re}{\operatorname{Re}}
\renewcommand{\Im}{\operatorname{Im}}
\newcommand{\dd}{\mathrm{d}}
\newcommand{\I}{\mathrm{i}}
\newcommand{\bk}{\mathrm{k}}
\newcommand{\e}{\mathrm{e}}
\newcommand{\bpd}{\bar\partial}
\newcommand{\bR}{\mathbb{R}}
\newcommand{\bC}{\mathbb{C}}
\newcommand{\cM}{\mathscr{M}}
\newcommand{\cP}{\mathscr{P}}
\newcommand{\cW}{\mathscr{W}}
\newcommand{\cH}{\mathscr{H}}
\newcommand{\cL}{\mathcal{L}}
\newcommand{\cN}{\mathcal{N}}

\title[SMOOTH QUASI--EINSTEIN CR LIFTS]
{Smooth Quasi--Einstein Lifts of Locally Embeddable Three-Dimensional CR Manifolds}
\author{Masoud Ganji}
\address{Masoud Ganji\newline
\hspace*{\parindent}School of Science and Technology\newline
\hspace*{\parindent}University of New England\newline
\hspace*{\parindent}Armidale NSW 2351\newline
\hspace*{\parindent}Australia}
\email{mganjia2@une.edu.au}

\subjclass[2020]{Primary 32V05, 53B30; Secondary 35H20, 83C15}
\keywords{CR geometry, local embeddability, quasi--Einstein metric,
shearfree null congruence, contact operator, Petrov classification}

\hypersetup{
  pdftitle={Smooth Quasi--Einstein Lifts of Locally Embeddable Three-Dimensional CR Manifolds},
  pdfauthor={Masoud Ganji},
  pdfsubject={Smooth Quasi--Einstein Lifts of Locally Embeddable Three-Dimensional CR Manifolds}
}

\begin{document}

\begin{abstract}
We prove a $C^\infty$ analogue of the real-analytic quasi--Einstein
lifting theorem of Hill, Lewandowski and Nurowski
\cite[Theorem~2.5]{HLN2008}.
Every smooth, locally CR-embeddable, strictly pseudoconvex
three-dimensional CR manifold admits, for each real cosmological
constant, a local smooth Lorentzian quasi--Einstein lift whose
null congruence is geodesic, shearfree and twisting and whose
Weyl tensor has Petrov type~$\mathrm{II}$ or $\mathrm{D}$.
The result extends local lifting beyond the Sasakian setting of
the global constructions in \cite{GSS2020} and the local results
in \cite{GGS2026}.
The proof establishes smooth local solvability of the reduced
semilinear contact equation by combining an $L^\infty$ contraction
with subelliptic regularity.
\end{abstract}

\maketitle

\section{Introduction and main result}\label{sec:introduction}

Twisting shearfree congruences of null geodesics provide a direct link between
four-dimensional Lorentzian geometry and three-dimensional CR
geometry. This classical correspondence \cite{Rob,RT1} has been
developed through the study of shearfree null congruences and
Robinson structures \cite{Tafel1,Traut2,Arman2}.
Following \cite{AGS} and \cite[Lemma~2.1]{AGSS}, we describe
the shearfree condition by the Lie-derivative relation
\eqref{eq:shearfree}.
The connection between shearfree geometry and algebraic speciality
is expressed by the Goldberg--Sachs theorem
and its refinements \cite{GoldSa,GHN}.

The local lifting problem asks which CR structures admit Lorentzian
metrics satisfying the Einstein equations with a cosmological term
and aligned pure radiation. In this paper, \emph{quasi--Einstein}
means
\begin{equation}\label{eq:quasi-einstein}
 \Ric(g)=\Lambda g+\Phi\,\lambda\otimes\lambda,
 \qquad \Lambda\in\bR,
\end{equation}
where $\lambda$ is the pulled-back CR contact form and $\Phi$ is a
smooth real function on the spacetime.

For the metric construction, we recall the description used in
\cite{GSS2020}. Let $\bk$ be a nowhere-zero null generator of a
twisting shearfree congruence. Its local leaf space $M$ carries a
CR structure $(D,J)$: the rank-two
distribution $D$ is the projection of $\bk^\perp$ to $M$, and
$J$ is induced by the conformal structure on
$\bk^\perp/\langle\bk\rangle$, after a local orientation is
chosen \cite{RT1}.
Locally, this CR structure is represented by a pair
$(\mu,\lambda)$, where $\lambda$ is a real one-form and $\mu$
is a complex one-form satisfying
\[
 D=\ker\lambda,\qquad
 \mu(JV)=\I\mu(V)\quad(V\in D),\qquad
 \lambda\wedge\mu\wedge\bar\mu\ne0.
\]
We write $[(\mu,\lambda)]$ for the equivalence class of pairs
defining the same CR structure. Representatives are related by
\begin{equation}\label{eq:CR-equivalence}
 \lambda'=a\lambda,\qquad \mu'=b\mu+h\lambda,
 \qquad a>0,\quad b\ne0,
\end{equation}
where $a$ is real and $b,h$ are complex functions.

For one-forms $\omega,\gamma$, we use $\vee$ and juxtaposition to
denote the symmetrised tensor product
\[
 \omega\vee\gamma=\omega\gamma
   =\tfrac12(\omega\otimes\gamma+\gamma\otimes\omega).
\]
On a local product $\cM=M\times I$, let $r$ be the coordinate
on the real interval $I$, and pull back the base forms without
changing their notation. Every shearfree Lorentzian metric
associated with the CR manifold $(M,[(\mu,\lambda)])$ can be
written locally in the form \cite{RT1,HLN2008}
\begin{equation}\label{eq:metric-form}
 g=2\cP^2\left[
       \mu\bar\mu+
       \lambda\left(\dd r+\cW\mu+\overline{\cW}\bar\mu+\cH\lambda\right)
       \right],
\end{equation}
where $\cP>0$ and $\cH$ are smooth real functions on $\cM$,
and $\cW$ is a smooth complex function.
Any metric in this family is called a \emph{lift} of the underlying
CR structure \cite{GSS2020}.
The vertical null field $\bk=\partial_r$ generates its shearfree
congruence and satisfies $\bk^\flat=\cP^2\lambda$.
Strict pseudoconvexity gives $\lambda\wedge\dd\lambda\ne0$,
so this congruence is twisting.

The relation between the field equations and local CR embeddability
was established in \cite{LNT} and developed systematically in
\cite{HLN2008}. The resulting curvature reduction determines
the fibre dependence of $\cP,\cW,\cH$ and leaves a system of
equations on the CR manifold. The real-analytic lifting theorem
\cite[Theorem~2.5]{HLN2008} asserts that every real-analytic,
strictly pseudoconvex three-dimensional CR manifold admits a local
quasi--Einstein lift of Petrov type~$\mathrm{II}$ or $\mathrm{D}$ for each prescribed
$\Lambda\in\bR$. The extension of this theorem to sufficiently
smooth embeddable CR structures was conjectured in
\cite[Remark~3.25]{HLN2008}. The missing step was the local
existence of a positive smooth solution of the reduced scalar
curvature equation. The main result of this paper supplies that
step and proves the lifting theorem in the
$C^\infty$, locally embeddable category.

The distinction between the real-analytic and smooth categories
is substantial. A smooth abstract
three-dimensional CR structure need not be locally embeddable;
local CR embeddability is therefore an assumption of our theorem.
Embedding data solve the first-order part of the curvature system,
but the nonlinear scalar equation still requires an independent
smooth existence argument.

Our earlier work proved global and local smooth lifting results
under Sasakian hypotheses. In \cite[Theorem~1.1]{GSS2020}, we
established global quasi--Einstein lifts of Petrov type~$\mathrm{II}$ or $\mathrm{D}$
for a class of smooth Sasakian hypersurfaces over the complex
plane, under asymptotic conditions on their defining potentials
and with $\Lambda>0$. The same paper also obtained lifting
results for tubular CR structures with arbitrary $\Lambda$.
In our subsequent work \cite[Corollary~6.2]{GGS2026}, we proved
local smooth lifting for regular Sasakian three-manifolds for
every real cosmological constant. These results established smooth
lifting in the Sasakian setting, where symmetry permits an
elliptic reduction on a two-dimensional quotient. The present
theorem extends the local existence result to every smooth,
locally embeddable, strictly pseudoconvex three-dimensional CR
structure, without imposing a Sasakian or CR symmetry hypothesis.

We retain the metric form \eqref{eq:metric-form} and use the
curvature reduction recalled in Section~\ref{subsec:curvature-reduction}.
Local CR embedding data supply the first-order curvature constraint.
The remaining scalar equation acts on the three-dimensional CR
manifold. Its principal part is a contact sum of squares with a
smooth horizontal drift, and its nonlinear terms have cubic and
inverse-cubic dependence on the unknown. The contact bracket
condition supplies subelliptic regularity, while the inverse-cubic
term makes positivity essential to the existence argument.

The key analytic observation is that the local fundamental
solution in \cite{BBMP2017} has integral $O(\varepsilon^2)$
over a small control ball. This estimate yields a contraction
in $L^\infty$ around any prescribed positive constant and keeps
the solution uniformly away from zero.
The local Sobolev estimates of \cite{RS1976} then bootstrap the
bounded distributional solution to $C^\infty$.
After shrinking the base neighbourhood, the scalar coefficient can
be chosen uniformly as close as desired to the prescribed constant,
relative to the normalised coframe. The argument imposes no sign
restriction on the cosmological constant. Reconstructing the metric
from this solution and the embedding data gives the following
smooth local lifting theorem.

\Needspace{24\baselineskip}
\begin{theorem}\label{thm:main}
Let $(M^3,D,J)$ be a smooth, locally embeddable, strictly
pseudoconvex CR manifold. Fix $x_0\in M$, $\Lambda\in\bR$,
$p_*>0$ and $0<\delta<p_*/2$.
There exist a neighbourhood $U$ of $x_0$, a normalised CR coframe
$(\mu,\lambda)$ on $U$ with $\mu=\dd\zeta$ for a CR function $\zeta$,
and an open interval $I\ni0$, $I\subset(-\pi,\pi)$, with the following
properties.

There are smooth functions
\[
 p:U\longrightarrow(0,\infty),\qquad
 m:U\longrightarrow\bC\setminus\{0\},
\]
and smooth functions $\cP>0$, $\cH$, $\Phi$ (real) and $\cW$ (complex)
on $\cM=U\times I$ such that the Lorentzian metric
\eqref{eq:metric-form} satisfies
\eqref{eq:quasi-einstein}. Moreover,
\begin{equation}\label{eq:p-control-main}
 |p-p_*|\le\delta\quad\hbox{on }U,\qquad
 \cP(x,r)=\frac{p(x)}{\cos(\frac{r}{2})}.
\end{equation}
In the null coframe \eqref{eq:null-coframe},
the Weyl scalars satisfy
\begin{equation}\label{eq:Psi2-main}
 \Psi_0=\Psi_1=0,\qquad
 \Psi_2=\frac{(1+\e^{\I r})^3}{2p^6}\,m\ne0.
\end{equation}
Consequently, the metric has Petrov type~$\mathrm{II}$ or $\mathrm{D}$ at every point of
$\cM$.
\end{theorem}

The range estimate for $p$ refers to the normalised coframe.
The radiation coefficient $\Phi$ is determined by the metric and
may vanish. Section~\ref{sec:geometry} fixes the
geometric conventions and records the curvature reduction.
Section~\ref{sec:analysis} explains the Sobolev notation and proves
smooth local semilinear solvability. Section~\ref{sec:proof} gives
the complete proof of Theorem~\ref{thm:main}.

\section{CR geometry and curvature reduction}\label{sec:geometry}

We first recall the shearfree relation, then fix the CR coframe
and curvature conventions used in the proof.

\begin{definition}\label{def:shearfree}
On a Lorentzian $n$-manifold $(\cM,g)$, a \emph{null congruence}
is a foliation by null curves. It is \emph{shearfree} if a
nowhere-zero null vector field $\bk$ tangent to its leaves satisfies
\begin{equation}\label{eq:shearfree}
 \cL_{\bk}g=f g+\bk^\flat\vee\vartheta,\qquad
 \bk^\flat=g(\bk,\cdot),
\end{equation}
for a smooth real function $f$ and a smooth real one-form $\vartheta$.
\end{definition}

Relation \eqref{eq:shearfree} is equivalent to the local flow of
$\bk$ preserving $\bk^\perp$ and the conformal class of the
degenerate metric $g|_{\bk^\perp\times\bk^\perp}$
\cite[Lemma~2.1]{AGSS}; see also \cite{AGS,GGS2026}.
It implies that the integral curves of $\bk$ are geodesics,
possibly with a nonaffine parametrisation
\cite[Proposition~2.4]{AGSS}. Thus such a foliation is also called
a \emph{geodesic shearfree null congruence}.
In dimension four, it is twisting if
$\bk^\flat\wedge\dd\bk^\flat\ne0$.

\subsection{CR structures}

\begin{definition}
A three-dimensional CR manifold is a triple $(M,D,J)$, where $M$
is a smooth real three-dimensional manifold, $D\subset TM$ is a
rank-two distribution, and $J:D\to D$ is a smooth endomorphism
with $J^2=-\operatorname{Id}$.
\end{definition}

The complex line distributions $D^{1,0}$ and $D^{0,1}$ are the
$+\I$ and $-\I$ eigenspaces of $J$ in $D\otimes\bC$, respectively.
Thus $D^{1,0}$ is locally spanned by $V-\I JV$, for any nowhere-zero
local real section $V$ of $D$. In CR dimension one, formal
integrability is automatic. We assume strict pseudoconvexity:
for every such $V$, one has $[V,JV]\notin D$ at each point.
Equivalently, every nowhere-zero local real one-form $\lambda$
annihilating $D$ satisfies
\[
 \lambda\wedge\dd\lambda\ne0.
\]
We choose its sign so that $\dd\lambda(\,\cdot\,,J\,\cdot\,)$
is positive definite on $D$, and use this coorientation throughout.

The CR structure can also be represented by a real one-form
$\lambda$ and a complex one-form $\mu$ such that
\[
 D=\ker\lambda,\qquad
 \mu(JV)=\I\mu(V)\quad(V\in D),\qquad
 \lambda\wedge\mu\wedge\bar\mu\ne0.
\]
The dual frame to $(\mu,\bar\mu,\lambda)$ is written
$(\partial,\bpd,\partial_0)$; $\partial_0$ is real.

\begin{definition}
A smooth function $f:M\to\bC$ is a CR function if $\bpd f=0$.
The CR manifold is locally CR-embeddable if near every point there
are CR functions $\zeta,\eta$ for which
$(\zeta,\eta):M\to\bC^2$ is a smooth embedding.
\end{definition}

The tangential equation $\bpd f=0$ is independent of the choice of
representative. On a locally embeddable neighbourhood, choose a CR
function $\zeta$ with $\dd\zeta|_D\ne0$ and put $\mu=\dd\zeta$.
Strict pseudoconvexity permits a compatible choice of $\lambda$
for which
\begin{equation}\label{eq:normalized-structure}
 \dd\mu=0,\qquad
 \dd\lambda=\I\mu\wedge\bar\mu+
                 (c\mu+\bar c\,\bar\mu)\wedge\lambda.
\end{equation}
Indeed, the nonzero horizontal Levi coefficient of any compatible
contact form can be normalised to one by a positive rescaling.
The remaining terms in its exterior derivative determine the smooth
complex function $c$. A coframe satisfying
\eqref{eq:normalized-structure} is called a
\emph{normalised CR coframe}. All subsequent curvature formulae
are used in this coframe.

Evaluating the exterior derivatives on the dual frame gives
\begin{equation}\label{eq:commutators}
 [\partial,\bpd]=-\I\partial_0,\qquad
 [\partial,\partial_0]=-c\partial_0,\qquad
 [\bpd,\partial_0]=-\bar c\,\partial_0.
\end{equation}
There are no horizontal bracket components because
$\dd\mu=\dd\bar\mu=0$.
The identity $\dd^2\lambda=0$ gives
\begin{equation}\label{eq:c-integrability}
 \partial\bar c=\bpd c.
\end{equation}
%

\subsection{Null coframes and curvature reduction}
\label{subsec:null-coframes}
\label{subsec:curvature-reduction}

On a local product $\cM=M\times I$, all base forms and functions
are pulled back without changing their notation. For every metric
\eqref{eq:metric-form}, the field $\bk=\partial_r$ is null,
$\bk^\flat=\cP^2\lambda$, and \eqref{eq:shearfree} holds with
\[
 \begin{aligned}
 f&=2\partial_r\log\cP,\\
 \vartheta&=2\big[
    (\partial_r\cW)\mu+
    (\partial_r\overline{\cW})\bar\mu+
    (\partial_r\cH)\lambda\big].
 \end{aligned}
\]
Thus the vertical null congruence is geodesic and shearfree.
This family depends on the CR structure: a change of CR coframe can be absorbed
by the metric coefficients and the choice of fibre coordinate.

For \eqref{eq:metric-form}, define the complex null coframe
\begin{equation}\label{eq:null-coframe}
 \begin{aligned}
 \theta^1&=\cP\mu,&\qquad \theta^2&=\cP\bar\mu,\\
 \theta^3&=\cP\lambda,&
 \theta^4&=\cP\left(\dd r+\cW\mu+\overline{\cW}\bar\mu+\cH\lambda\right).
 \end{aligned}
\end{equation}
Its dual frame is
\begin{equation}\label{eq:null-frame}
 \begin{aligned}
 e_1&=\cP^{-1}(\partial-\cW\partial_r),&
 e_2&=\cP^{-1}(\bpd-\overline{\cW}\partial_r),\\
 e_3&=\cP^{-1}(\partial_0-\cH\partial_r),&
 e_4&=\cP^{-1}\partial_r.
 \end{aligned}
\end{equation}
Thus $g=2(\theta^1\theta^2+\theta^3\theta^4)$ has Lorentzian
signature $(3,1)$. We write $\Ric_{ij}=\Ric(e_i,e_j)$ and use
the Ricci and Weyl conventions of \cite{HLN2008}, with Weyl
scalars $\Psi_0,\ldots,\Psi_4$ relative to this frame.

The curvature reduction in
\cite[Proposition~3.7 and Theorem~3.8]{HLN2008} applies in this
null coframe.
Integration along the fibre introduces a positive real base
function $p$, a real phase $s$ and complex base functions $t,m$,
all independent of $r$.
The function $t$ parametrises the coefficient $\cW$,
and $m$ enters the coefficient $\cH$.
On a fibre interval where $\cos(\frac{r+s}{2})>0$, the general
formulae for $\cP$ and $\cW$ are
\begin{equation}\label{eq:PW-general}
 \begin{aligned}
 \cP&=\frac{p}{\cos(\frac{r+s}{2})},\\
 \cW&=\I\big(c+2\partial\log p-t\big)\e^{-\I(r+s)}
       +\I\big(c+2\partial\log p-2t\big)+\partial s.
 \end{aligned}
\end{equation}
Thus $t$ is an auxiliary metric unknown; the function $c$ is
already fixed by the normalised CR coframe.

\Needspace{12\baselineskip}
\begin{theorem}[\cite{HLN2008}]\label{thm:published-curvature}
For the normalised coframe \eqref{eq:normalized-structure} with
$\mu=\dd\zeta$, and a fixed $\Lambda\in\bR$, the curvature conditions
\[
 \Ric_{44}=\Ric_{24}=\Ric_{22}=0,\qquad
 \Ric_{12}=\Ric_{34}=\Lambda
\]
reduce to
\begin{equation}\label{eq:t-curvature}
 \partial t+(c-t)t=0
\end{equation}
and
\begin{equation}\label{eq:scalar-curvature-general}
 \begin{aligned}
 \Big[\partial\bpd+\bpd\partial+\bar c\,\partial+c\bpd
       +\tfrac12|c|^2+\tfrac34(\partial\bar c+\bpd c)
      -\tfrac32(\partial\bar t&+\bpd t+t\bar t)\Big]p\\
       &=\frac{m+\bar m}{p^3}+\frac23\Lambda p^3.
 \end{aligned}
\end{equation}
Conversely, smooth base solutions $p>0,t,m$ of these equations
give such a metric using \eqref{eq:PW-general} and the formula for
$\cH$ in \cite[eq.~(3.28)]{HLN2008}.
\end{theorem}

The equation for $t$ is a first-order nonlinear PDE on $M$,
where $\partial$ is the complex vector field defined above.
The ordinary differential equations in the fibre variable $r$
have already been integrated in \eqref{eq:PW-general}.
We now choose the identically zero solution $t\equiv0$.
This choice works for every smooth coframe coefficient $c$.
Independently, the change of fibre coordinate $r\mapsto r+s$
normalises the phase to $s=0$.
We use $t=s=0$ throughout the construction below.
The $t$-dependent terms in
\eqref{eq:scalar-curvature-general} then vanish, leaving
\eqref{eq:scalar-curvature}. The additional mixed curvature
condition $\Ric_{13}=0$ is recorded in
Proposition~\ref{prop:curvature-reduction} below.

Fix $\Lambda\in\bR$, and let $p>0$ be real and $m$ complex
on the base, independent of $r$.
On $I\subset(-\pi,\pi)$ set
\begin{equation}\label{eq:PW}
 \cP=\frac{p}{\cos(\frac{r}{2})},\qquad
 \cW=\I\big(c+2\partial\log p\big)(1+\e^{-\I r}),
\end{equation}
and define the base functions $Q$ (complex) and $h$ (real) by
\begin{equation}\label{eq:Q}
 \begin{aligned}
 Q={}&\frac{3m+\bar m}{p^4}+\frac23\Lambda p^2
 +\frac{2(\partial p)(\bpd p)
       -p(\partial\bpd p+\bpd\partial p)}{2p^2}
       -\frac{\I}{2}\partial_0\log p-\bpd c.
 \end{aligned}
\end{equation}
\begin{equation}\label{eq:h}
 \begin{aligned}
 h={}3\frac{m+\bar m}{p^4}+2\Lambda p^2+\frac{2(\partial p)(\bpd p)
       -p(\partial\bpd p+\bpd\partial p)}{p^2}-2\bpd c.
 \end{aligned}
\end{equation}
Then
\begin{equation}\label{eq:calH}
 \cH=\frac{m}{p^4}\e^{2\I r}
       +\frac{\bar m}{p^4}\e^{-2\I r}
       +Q\e^{\I r}+\bar Q\e^{-\I r}+h.
\end{equation}

\Needspace{6\baselineskip}
\begin{proposition}[\cite{HLN2008}]
\label{prop:curvature-reduction}
For the normalised coframe \eqref{eq:normalized-structure} and coefficients
\eqref{eq:PW}--\eqref{eq:calH} with $t=s=0$, the following assertions hold.
\begin{enumerate}
\item The metric \eqref{eq:metric-form} satisfies
\[
 \Ric_{44}=\Ric_{22}=\Ric_{24}=0,\qquad
 \Ric_{12}+\Ric_{34}=2\Lambda.
\]
\item The conditions $\Ric_{12}=\Ric_{34}=\Lambda$ hold if
and only if $p$ satisfies
\begin{equation}\label{eq:scalar-curvature}
 \begin{aligned}
 \Big[\partial\bpd+\bpd\partial+\bar c\,\partial+c\bpd
      +\tfrac12|c|^2
      +\tfrac34(\partial\bar c+\bpd c)\Big]p
   =\frac{m+\bar m}{p^3}+\frac23\Lambda p^3.
 \end{aligned}
\end{equation}
\item Assuming \eqref{eq:scalar-curvature}, the condition $\Ric_{13}=0$ is equivalent to
\begin{equation}\label{eq:m-eqn}
 \partial m+3cm=0.
\end{equation}
\item If \eqref{eq:scalar-curvature} and \eqref{eq:m-eqn} hold, then the metric satisfies
\eqref{eq:quasi-einstein} for a smooth real function $\Phi$, and
\begin{equation}\label{eq:Psi2-reconstruction}
 \Psi_0=\Psi_1=0,\qquad
 \Psi_2=\frac{(1+\e^{\I r})^3}{2p^6}\,m.
\end{equation}
\end{enumerate}
\end{proposition}

\begin{proof}
We apply Theorem~\ref{thm:published-curvature} with the solution
$t\equiv0$ of \eqref{eq:t-curvature} and the phase $s=0$.
The stated identities follow from
\cite[eqs.~(3.23)--(3.29) and (3.35)--(3.37), Remark~3.17]{HLN2008}
with these choices.
In that reference, equations~(3.23)--(3.27)
give \eqref{eq:PW} and the first three vanishing Ricci components.
Equation~(3.28) gives \eqref{eq:Q}--\eqref{eq:calH}
and the trace condition.
Equation~(3.29) gives \eqref{eq:scalar-curvature}, and equation~(3.35) gives
\eqref{eq:m-eqn}. The Weyl scalar formula is equation~(3.36);
the tensor Ricci equation is explained in Remark~3.17.
\end{proof}

Once the conditions in the last part of
Proposition~\ref{prop:curvature-reduction} hold, only the component
$\Ric_{33}$ remains unconstrained.
Since $g_{33}=0$ and $\lambda(e_3)=\cP^{-1}$,
the radiation coefficient is
\begin{equation}\label{eq:Phi}
 \Phi=\cP^2\Ric_{33}.
\end{equation}
This identifies $\Phi$ without introducing an additional scalar equation.

\section{Local semilinear solvability for contact operators}
\label{sec:analysis}

We prove the smooth local existence statement needed for the reduced
scalar curvature equation. 

Let $\Omega\subset\bR^3$ be open, let $j\ge0$ be an integer,
and let $1<q<\infty$.
The notation $W^{j,q}_{\mathrm{loc}}(\Omega)$ denotes the
\emph{local Sobolev space}:
\begin{equation}\label{eq:sobolev-definition}
 u\in W^{j,q}_{\mathrm{loc}}(\Omega)
 \quad\Longleftrightarrow\quad
 \partial_x^\ell u\in L^q_{\mathrm{loc}}(\Omega)
 \quad\text{for every }|\ell|\le j.
\end{equation}
Here $\ell=(\ell_1,\ell_2,\ell_3)\in\mathbb N_0^3$,
$|\ell|=\ell_1+\ell_2+\ell_3$, and
$\partial_x^\ell
 =\partial_{x_1}^{\ell_1}\partial_{x_2}^{\ell_2}
   \partial_{x_3}^{\ell_3}$ is a weak derivative, understood in
the distributional sense.
The subscript $\mathrm{loc}$ means that these integrability
conditions hold on every relatively compact open set
$V\Subset\Omega$.
Equivalently, $u|_V\in W^{j,q}(V)$ for every such $V$.
Thus $W^{0,q}_{\mathrm{loc}}=L^q_{\mathrm{loc}}$, and
$W^{1,4}_{\mathrm{loc}}$ consists of functions whose values
and first weak coordinate derivatives are locally
fourth-power integrable.
These are ordinary coordinate Sobolev spaces.
The metric coefficient is always written $\cW$, whereas
the plain letter $W$ in \eqref{eq:sobolev-definition}
denotes a function space.
We choose $q=4$ because the coordinate domain has dimension three.
The Sobolev--Morrey embedding theorem
\cite[Theorem~4.12, Part~II]{AF2003}, applied locally on balls, gives
\[
 W^{1,q}_{\mathrm{loc}}(\Omega)
 \hookrightarrow C^{0,\alpha}_{\mathrm{loc}}(\Omega),
 \qquad \alpha=1-\frac{3}{q},\qquad 3<q<\infty.
\]
This means that every function in $W^{1,q}_{\mathrm{loc}}(\Omega)$
has a locally H\"older continuous representative equal to it almost
everywhere. In particular, taking $q=4$ gives
\[
 W^{1,4}_{\mathrm{loc}}(\Omega)
 \hookrightarrow C^{0,1/4}_{\mathrm{loc}}(\Omega).
\]
More explicitly, for balls $B\Subset B'\Subset\Omega$, this
representative satisfies
\[
 |u(x)-u(y)|
 \le C\,\|u\|_{W^{1,4}(B')}\,|x-y|^{1/4},
 \qquad x,y\in B,
\]
where $C$ depends on the balls but not on $u$.
Here $\Subset$ denotes compact containment.
Since the argument is local, no regularity assumption on
$\partial\Omega$ is required.
Thus the first coordinate Sobolev regularity gain yields a H\"older
continuous representative. The subsequent regularity gains,
together with the local product and composition rules, then yield
smoothness.

\begin{lemma}\label{lem:bootstrap}
Let $X,Y$ be smooth real vector fields on an open set
$\Omega\subset\bR^3$ such that
\begin{equation}\label{eq:contact-bracket}
 \Span\{X,Y,[X,Y]\}=T\Omega.
\end{equation}
For smooth real functions $b_1,b_2$ set
\begin{equation}\label{eq:contact-operator}
 \cL=X^2+Y^2+X_0,\qquad X_0=b_1X+b_2Y.
\end{equation}
Let $\mathcal J\subset\bR$ be an open interval and
$\cN\in C^\infty(\Omega\times\mathcal J;\bR)$.
Suppose $u\in L^\infty_{\mathrm{loc}}(\Omega)$ solves
$\cL u=\cN(x,u)$ in distributions. Assume that on each relatively
compact subdomain of $\Omega$, the essential range of $u$ lies
in a compact subinterval of $\mathcal J$.
Then $u$ has a smooth representative.
\end{lemma}

\begin{proof}
Assign weight one to $X,Y$ and weight two to $X_0$.
The weighted bracket condition has step two by
\eqref{eq:contact-bracket}, including at zeros of $X_0$.
The regularity estimates in \cite[Theorems~16(a) and~18]{RS1976}
therefore give, for every integer $j\ge0$,
\begin{equation}\label{eq:sobolev-gain}
 u\in L^4_{\mathrm{loc}},\quad
 \cL u\in W^{j,4}_{\mathrm{loc}}
 \quad\Longrightarrow\quad
 u\in W^{j+1,4}_{\mathrm{loc}}.
\end{equation}
These are ordinary coordinate Sobolev spaces, corresponding to
the integer-order Bessel-potential spaces in that reference.
The differentiability gain is $2/2=1$, and the result applies
to the equation in the weak sense.

On any relatively compact neighbourhood on which the compact-range
hypothesis holds, both $u$ and $\cN(x,u)$ are bounded.
Taking $j=0$ in \eqref{eq:sobolev-gain} gives
$u\in W^{1,4}_{\mathrm{loc}}$. The Sobolev--Morrey embedding
above therefore gives a representative in
$C^{0,1/4}_{\mathrm{loc}}(\Omega)$.
Using a smooth cut-off in the second variable supported in
$\mathcal J$, extend the nonlinearity smoothly to all real values
of that variable, leaving it unchanged near this compact range.
The local Sobolev composition
rules give
\begin{equation}\label{eq:composition}
 u\in W^{j,4}_{\mathrm{loc}}
 \quad\Longrightarrow\quad
 \cN(x,u)\in W^{j,4}_{\mathrm{loc}},
 \qquad j\ge1.
\end{equation}
For $j=1$ this is the weak chain rule; for higher $j$ it follows
from the higher-order composition rule and the Sobolev algebra
property, since $4>3$.
Alternating \eqref{eq:composition} with \eqref{eq:sobolev-gain}
gives $u\in W^{j,4}_{\mathrm{loc}}$ for every $j$.
Sobolev embedding yields a smooth representative.
\end{proof}

Lemma~\ref{lem:bootstrap} shows that a bounded distributional
solution is smooth provided its values remain in a compact
subinterval of $\mathcal J$.
It remains to construct such a solution with control of its range.
The following theorem does so by a contraction on a small control
ball. For a positive reference value, this control also ensures
the positivity required by the reduced curvature equation
\eqref{eq:scalar-curvature}.

\begin{theorem}\label{thm:semilinear}
Assume \eqref{eq:contact-bracket}--\eqref{eq:contact-operator}
near $x_0\in\bR^3$.
Let $\mathcal J\subset\bR$ be an open interval and let
$\cN\in C^\infty(\Omega\times\mathcal J;\bR)$.
For every $p_*\in\mathcal J$ and every $\delta>0$ with
$[p_*-\delta,p_*+\delta]\Subset\mathcal J$, there is a neighbourhood
$V$ of $x_0$ and a smooth solution $p$ satisfying
\begin{equation}\label{eq:semilinear}
 \cL p=\cN(x,p)\quad\hbox{on }V,\qquad |p-p_*|\le\delta.
\end{equation}
If $p_*>0$ and $\delta<p_*/2$, this solution is positive.
\end{theorem}

\begin{proof}
Work in a fixed relatively compact coordinate neighbourhood where
the coefficients are smooth and the bracket condition holds.
The smooth vector fields satisfy the hypotheses of
\cite[Theorem~4.8]{BBMP2017}, which gives a local fundamental
solution $\gamma(x,y)$
on a sufficiently small fixed neighbourhood $U_0$ of $x_0$ such that
\begin{equation}\label{eq:fundamental}
 \cL_x\gamma(x,y)=-\delta_y,\qquad
 |\gamma(x,y)|\le C\phi_2(x,y).
\end{equation}
We may take $\gamma$ real by taking its real part.
Let $d$ be the weighted control distance used in that reference,
associated with $X,Y,X_0$ and the weights specified above.
We use this distance throughout.
For the fixed upper integration radius $\rho_0>0$ used in that
construction and for $\nu>0$, the comparison kernels are
\[
 \phi_\nu(x,y)=\int_{d(x,y)}^{\rho_0}
       \frac{t^{\nu-1}}{|B_d(x,t)|}\,\dd t.
\]
All balls considered below lie within the fixed neighbourhood on
which these estimates hold. Since the Euclidean dimension is three and
$2<3$, the kernel estimate in
\cite[Corollary~3.4]{BBMP2017} gives
\begin{equation}\label{eq:kernel-integral}
 \sup_x\int_{d(x,y)<\rho}\phi_2(x,y)\,\dd y\le C\rho^2
\end{equation}
for small $\rho$, uniformly on a smaller fixed neighbourhood.
Set $B_\varepsilon=B_d(x_0,\varepsilon)$, with
$\overline{B_\varepsilon}\subset U_0$.
For $x,y\in B_\varepsilon$, the control-distance triangle estimate
gives $d(x,y)\le C\varepsilon$.
Applying \eqref{eq:kernel-integral} at this radius and using
\eqref{eq:fundamental} gives
\begin{equation}\label{eq:small-kernel}
 \sup_{x\in B_\varepsilon}
       \int_{B_\varepsilon}|\gamma(x,y)|\,\dd y
       \le C\varepsilon^2.
\end{equation}
Here $C$ is independent of $\varepsilon$:
the fundamental solution is fixed before its source support
is reduced.

Choose a smooth cut-off $0\le\chi_\varepsilon\le1$, supported in
$B_\varepsilon$ and equal to one on $B_{\varepsilon/2}$.
These cut-offs exist because the control distance induces the
coordinate topology and the closed smaller ball is compactly
contained in the larger one. No derivative bounds for the cut-offs
enter the contraction.
On the complete closed ball
\[
 \mathcal B_\delta
  =\{v\in L^\infty(B_\varepsilon;\bR):\|v\|_\infty\le\delta\},
\]
define
\begin{equation}\label{eq:fixed-point}
 \mathcal T(v)(x)
  =-\int_{B_\varepsilon}\gamma(x,y)\chi_\varepsilon(y)
           \cN(y,p_*+v(y))\,\dd y.
\end{equation}
Choose finite constants $M_0,M_1$ bounding $|\cN(x,\xi)|$ and
$|\partial_\xi\cN(x,\xi)|$ on a fixed compact coordinate neighbourhood
times $[p_*-\delta,p_*+\delta]$.
Equation~\eqref{eq:small-kernel} and the mean-value theorem give
\begin{align}
 \|\mathcal T(v)\|_\infty
     &\le C\varepsilon^2M_0,\label{eq:self-map}\\
 \|\mathcal T(v_1)-\mathcal T(v_2)\|_\infty
     &\le C\varepsilon^2M_1\|v_1-v_2\|_\infty.
       \label{eq:contraction}
\end{align}
Take $\varepsilon$ small enough that
$C\varepsilon^2M_0\le\delta$ and $C\varepsilon^2M_1<1$.
Banach's fixed-point theorem gives
$v\in\mathcal B_\delta$ with $\mathcal T(v)=v$.
Set $p=p_*+v$.

Put $f=\chi_\varepsilon\cN(\cdot,p)$ on $B_\varepsilon$
and extend it by zero to $U_0$.
Let $\cL^*$ be the formal adjoint for coordinate Lebesgue measure.
For $\psi\in C_c^\infty(B_\varepsilon)$, the
fundamental-solution identity gives
\begin{align*}
 \int_{B_\varepsilon}v(x)\cL^*\psi(x)\,\dd x
 &=-\int_{B_\varepsilon}f(y)
       \left(\int_{U_0}\gamma(x,y)\cL^*\psi(x)\,\dd x\right)\dd y\\
 &=\int_{B_\varepsilon}f(y)\psi(y)\,\dd y.
\end{align*}
Fubini's theorem applies: integrating \eqref{eq:small-kernel}
in $x$ gives absolute integrability of the double integral,
since $f$ and $\cL^*\psi$ are bounded and supported in
$B_\varepsilon$.
Thus $\cL v=f$ in distributions. Since $\cL p_*=0$,
\[
 \cL p=\cN(x,p)\quad\hbox{on }B_{\varepsilon/2}.
\]

The essential range of $p$ lies in the compact interval
$[p_*-\delta,p_*+\delta]\Subset\mathcal J$.
Lemma~\ref{lem:bootstrap} gives a smooth representative on
$B_{\varepsilon/2}$. Its values still lie in the same closed
interval, by continuity and the almost-everywhere bound.
Take $V=B_{\varepsilon/2}$, or a smaller coordinate neighbourhood
of $x_0$. For $p_*>0$ and $\delta<p_*/2$,
$p\ge p_*-\delta>p_*/2$.
\end{proof}

\begin{remark}\label{rem:analytic-scope}
The construction is an interior local existence argument and imposes
no boundary data. Estimate~\eqref{eq:small-kernel} gives an
$O(\varepsilon^2)$ bound for the $L^\infty$ norm of the underlying
linear integral operator; the nonlinear map $\mathcal T$ is a
contraction by \eqref{eq:contraction}.
The chosen fundamental solution and cut-off select one fixed point;
uniqueness of that fixed point does not imply uniqueness of all
local solutions of \eqref{eq:semilinear}.
\end{remark}

\section{Proof of the main theorem}\label{sec:proof}

\begin{proof}[Proof of Theorem~\ref{thm:main}]
Fix $x_0\in M$, $\Lambda\in\bR$, $p_*>0$ and
$0<\delta<p_*/2$.
Choose a smooth local CR embedding $(\zeta,\eta)$ into $\bC^2$
on a coordinate neighbourhood $U_0$ of $x_0$.
After a complex linear change of the embedding coordinates and
shrinking $U_0$, arrange $\dd\zeta|_D\ne0$ and use the normalised
CR coframe \eqref{eq:normalized-structure} with $\mu=\dd\zeta$.
The dual frame and the function $c$ are those of
Section~\ref{sec:geometry}.
We use the choice $t=s=0$ fixed in
Section~\ref{subsec:curvature-reduction}.

We first choose $m$ so that the first-order constraint
\eqref{eq:m-eqn} holds and the Weyl scalar $\Psi_2$ will be
nowhere zero.
Since $\mu=\dd\zeta$, one has $\partial_0\zeta=0$.
The differential of the embedding is injective, so
$\partial_0\eta\ne0$; otherwise the nonzero real vector
$\partial_0$ would annihilate both embedding coordinates.
Define
\begin{equation}\label{eq:m-def}
 m=(\partial_0\bar\eta)^3.
\end{equation}
The CR condition $\bpd\eta=0$ implies
$\partial\bar\eta=0$.
Consequently, \eqref{eq:commutators} gives
\[
 \partial(\partial_0\bar\eta)
   =[\partial,\partial_0]\bar\eta
       +\partial_0(\partial\bar\eta)
   =-c\,\partial_0\bar\eta.
\]
Thus $m$ is smooth and nowhere zero, and
$\partial m=-3cm$, proving \eqref{eq:m-eqn}.

It remains to solve the real scalar equation
\eqref{eq:scalar-curvature} for $p>0$ with this fixed $m$.
Define real vector fields
\begin{equation}\label{eq:XY}
 X=\partial+\bpd,\qquad
 Y=\I(\partial-\bpd),\qquad
 \partial=\tfrac12(X-\I Y).
\end{equation}
Then
\begin{equation}\label{eq:real-identities}
 \begin{aligned}
 \partial\bpd+\bpd\partial&=\tfrac12(X^2+Y^2),\\
 c\bpd+\bar c\,\partial&=\Re(c)X-\Im(c)Y.
 \end{aligned}
\end{equation}
By \eqref{eq:commutators}, $[X,Y]=-2\partial_0$.
Thus $X,Y,[X,Y]$ span the tangent space at every point of $U_0$.
The function
\begin{equation}\label{eq:V}
 V=\tfrac12|c|^2+\tfrac34(\partial\bar c+\bpd c)
\end{equation}
is real, since $\overline{\partial\bar c}=\bpd c$.
Multiplying \eqref{eq:scalar-curvature} by two and moving its
zeroth-order linear term to the right gives
\begin{equation}\label{eq:scalar-real}
 \big[X^2+Y^2+2\Re(c)X-2\Im(c)Y\big]p
 =2(m+\bar m)p^{-3}+\tfrac43\Lambda p^3-2Vp.
\end{equation}
This has the contact operator \eqref{eq:contact-operator}, with
$b_1=2\Re(c)$ and $b_2=-2\Im(c)$, and the real nonlinearity
\begin{equation}\label{eq:curvature-nonlinearity}
 \cN(x,\xi)=2\big(m(x)+\bar m(x)\big)\xi^{-3}
              +\tfrac43\Lambda\xi^3-2V(x)\xi,\qquad \xi>0.
\end{equation}
Hence $\cN\in C^\infty(U_0\times(0,\infty);\bR)$.
Applying Theorem~\ref{thm:semilinear} with
$\mathcal J=(0,\infty)$ gives a neighbourhood $U\subset U_0$
of $x_0$ and a smooth solution $p>0$ of
\eqref{eq:scalar-curvature} with $|p-p_*|\le\delta$.
The prescribed value of $\Lambda$ enters only as a smooth
coefficient of $\cN$, so the argument imposes no restriction on
its sign.

Choose an open interval $I\ni0$ with $I\subset(-\pi,\pi)$,
and define $\cP,\cW,Q,h,\cH$ by
\eqref{eq:PW}--\eqref{eq:calH}.
Since $p>0$ and $\cos(\frac{r}{2})>0$ on $I$, these coefficients
are smooth and $\cP>0$ on $U\times I$.
The estimate for $p$ and the definition of $\cP$ give
\eqref{eq:p-control-main}.
Both reduced equations \eqref{eq:scalar-curvature} and
\eqref{eq:m-eqn} now hold.
Proposition~\ref{prop:curvature-reduction} therefore gives
\eqref{eq:quasi-einstein}, with the smooth real coefficient
$\Phi=\cP^2\Ric_{33}$, and
\[
 \Psi_0=\Psi_1=0,\qquad
 \Psi_2=\frac{(1+\e^{\I r})^3}{2p^6}\,m\ne0.
\]
Indeed, $m$ is nowhere zero, $p>0$, and
$1+\e^{\I r}\ne0$ on $I$.
The metric therefore has Petrov type~$\mathrm{II}$ or $\mathrm{D}$
at every point of $U\times I$.
This completes the proof.
\end{proof}



\end{document}